\documentclass{SCAE}
\numberwithin{equation}{section}
\newcommand{\REFEQN}[1] { \begin{equation}\label{#1} }
\newcommand{\ENDEQN}{\end{equation}}
\newcommand{\REFTHM}[1] { \begin{theorem}\label{#1} }
\newcommand{\ENDTHM}{\end{theorem}}
\newcommand{\REFPROP}[1]{\begin{proposition}\label{#1} }
\newcommand{\ENDPROP}{\end{proposition} }
\newcommand{\REFLEM}[1]{\begin{lemma}\label{#1} }
\newcommand{\ENDLEM}{\end{lemma} }
\newcommand{\REFCOR}[1]{\begin{corollary}\label{#1} }
\newcommand{\ENDCOR}{\end{corollary} }
\newcommand{\Ref}[1]{(\ref{#1})}

\def\R{{\mathbb R}}

\def\C{{\mathbb C}}

\def\D{{\mathbb D}}
\def\Z{{\mathbb Z}}

\def\N{{\mathbb N}}

\def\aa{{\bf a}}
\def\ww{{\bf w}}

\def\cc{{\bf c}}

\def\00{{\bf 0}}

\def\cbar{{\widehat{\C}}}

\def\AAA{{\cal A}}

\def\CCC{{\cal C}}

\def\WWW{{\cal W}}

\def\MMM{{\cal M}}

\def\VVV{{\cal V}}

\def\CCC{{\cal C}}
\def\sp{Schwarzian primitive\ }
\def\al{\alpha}

\def\g{\gamma}

\def\De{\Delta}
\def\de{\delta}
\def\smm{\smallsetminus}
\def\ds{\displaystyle}
\def\lv{\left(\begin{matrix}}

\def\rv{\end{matrix}\right)}

\begin{document}

\Year{2016} %
\Month{January}
\Vol{59} %
\No{1} %
\BeginPage{1} %
\EndPage{XX} %
\AuthorMark{Cui G Z {\it et al.}}
\ReceivedDay{November 17, 2014}
\AcceptedDay{January 22, 2015}
\PublishedOnlineDay{; published online January 22, 2016}
\DOI{10.1007/s11425-000-0000-0} 

\title[Rational maps as Schwarzian primitives]{Rational maps as Schwarzian primitives}{}


\author[1]{CUI GuiZhen}{}
\author[2]{GAO Yan}{Corresponding author.}
\author[3]{RUGH Hans Henrik}{}
\author[4]{TAN Lei}{}

\address[{\rm1}]{Academy of Mathematics and Systems Science,
Chinese Academy of Sciences, Beijing {\rm 100190}, China;}
\address[{\rm2}]{Department of Mathematics, Sichuan University, Chengdu {\rm 610065}, China;}
\address[{\rm3}]{B\^{a}timent 425, Facult\'{e} des Sciences d'Orsay, Universit\'{e} Paris-Sud, Paris {\rm91405}, France;}
\address[{\rm4}]{Facult\'{e} des sciences, LAREMA, Universit\'{e} d'Angers, Angers {\rm49045}, France.}
\Emails{ gzcui@math.ac.cn,
gyan@scu.edu.cn, Hans-Henrik.Rugh@math.u-psud.fr, tanlei@math.univ-angers.fr}\maketitle


 {\begin{center}
\parbox{14.5cm}{\begin{abstract}
We examine when a meromorphic
quadratic differential $\phi$ with prescribed poles is the Schwarzian derivative of a rational
map. We give a necessary and sufficient condition: In the Laurent series of $\phi$ around each pole $c$, the most singular term  should take the form $(1-d^2)/(2(z-c)^2)$, where $d$ is an integer, and then a certain determinant in the next $d$ coefficients should vanish. This condition can be optimized by neglecting some information on one of the poles (i.e. by only requiring it to be a double pole). The case $d=2$ was treated by   Eremenko  \cite{E}.

We show that a  geometric interpretation of our condition is that
the complex projective structure  induced by $\phi$ outside the poles  has a trivial holonomy group. This statement was suggested to us by W. Thurston in a private communication.

 Our work is  related to the problem to finding a rational map $f$ with a prescribed set of critical points,
 since the critical points of $f$ are precisely the poles of its Schwarzian derivative.

Finally, we study the pole-dependency of these Schwarzian derivatives. We show that, in the cubic case with simple critical points, an analytic dependency fails precisely when the
poles are displaced at the vertices of a regular ideal tetrahedron of the hyperbolic 3-ball.\vspace{-3mm}
\end{abstract}}\end{center}}



\renewcommand{\baselinestretch}{1.2}
\begin{center} \renewcommand{\arraystretch}{1.5}
{\begin{tabular}{lp{0.8\textwidth}} \hline \scriptsize
{\bf Citation:}\!\!\!\!&\scriptsize Cui G Z, Gao Y, Rugh H H, Tan L. \makeatletter\@titlehead.
Sci China Math, 2016, 59,
 doi:~\@DOI\makeatother\vspace{1mm}
\\
\hline
\end{tabular}}\end{center}

\baselineskip 11pt\parindent=10.8pt  \wuhao

\section{Introduction}

Let $U\subset\C$ be a domain. Recall that for a non-constant meromorphic function  $f:
U\to \cbar$, its Schwarzian derivative $S_f$ is a meromorphic
function defined by
$$S_f(z)=\left\{\begin{array}{ll} \dfrac{f'''(z)}{f'(z)}-\dfrac 32 \left(\dfrac{f''(z)}{f'(z)}\right)^2&\text{ if } f(z)\ne \infty\vspace{0.2cm} \\
\ds\lim_{w\to z} S_f(w) & \text{ if } f(z)= \infty\end{array}\right. .
$$

It is easily checked that the Schwarzian derivative of a M\"obius transformation is $0$.
There is a composition formula
\begin{equation}\label{composition-formula}
S_{u\circ v}(z)=S_u(v(z))\cdot v'(z)^2 + S_v(z).
\end{equation}
Let $\g$ be a  M\"obius transformation. We have then  $S_{\g\circ f}(z)\equiv S_f(z)$.
 On the other hand,  by considering $S_f(z)dz^2$ as a quadratic differential and by setting $z=\g(w)$, we have $S_{f\circ \g}(w)dw^2=S_f(z)dz^2$.
 I.e. the quadratic differential $S_f(z)dz^2$ is invariant if we pre-compose $f$ by a M\"{o}bius transformation.

 If $f$ is a rational map, then  the quadratic differential $S_f(z) dz^2$ is meromorphic on the Riemann sphere,  whose poles are located precisely at the critical points of $f$.

We give here a necessary and sufficient condition for a meromorphic
quadratic differential on the Riemann sphere to be the Schwarzian derivative of a rational
map.

For $d=1$ set $Y_1=0$. For every positive integer $d\ge 2$, let
 $Y_d(x_1,\cdots, x_{d-1})$ be the polynomial so that
$$ \left|\begin{matrix}
x_1 & 2\cdot 1 \cdot (1-d) & 0 & \cdots & 0\\
x_2 & x_1 & 2\cdot 2\cdot (2-d) & \cdots &0 \\
\vdots  & \vdots & \vdots & & \vdots\\
x_{d-1} & x_{d-2} & x_{d-3} & \cdots & 2(d-1)(-1)\\
x_d & x_{d-1} & x_{d-2} & \cdots & x_1 \end{matrix}\right| = N\cdot \Big(x_d-Y_d(x_1,\cdots, x_{d-1})\Big)$$
for some contant $N\ne 0$.

We will establish:

\begin{theorem}\label{except one}\quad
Let $\phi(z)dz^2$ be a  meromorphic quadratic differential on the Riemann sphere.

(a) If $\phi(z)=S_f(z)$ for some rational map $f$, then all poles of the quadratic differential are of order two,
and around each critical point  $c\in \C$ of $f$ with local degree $d_c\ge 2$, the local Laurent series expansion of $\phi(z)$ has the form
\REFEQN{local form} (z-c)^2\phi(z)= \dfrac{1-d_c^2}2 +a_1(z-c)+a_2(z-c)^2+\cdots, \text{ with }
 a_{d_c}=Y_{d_c}(a_1, a_2,\cdots, a_{d_c-1}).\ENDEQN
The same relation also holds  at $\infty$ (if it is a critical point of $f$) after change of coordinates by $1/z$.

(b) Conversely, assume that  $\phi(z)dz^2$ is a meromorphic quadratic differential, and for each pole $c$, except  possibly one, it takes the local expression \Ref{local form}  for some integer $d_c\ge 2$, furthermore the exceptional pole has order 2. Then $\phi(z)dz^2$ is the Schwarzian derivative of some rational function $f$, having the local expression \Ref{local form}  for every pole. The critical points of $f$ are precisely the poles of  $\phi(z)dz^2$, with  the integer $d_c$ being the local degree of $f$ at $c$.
\end{theorem}

   We remark that if we require in Theorem \ref{except one} (b) that $\phi(z)dz^2$ takes the local expression \Ref{local form} around each pole $c$, without exceptions, then this theorem is equivalent to Theorem \ref{local} below. The key point of the above theorem is that as long as  $\phi(z)dz^2$ takes the local expression \Ref{local form} around each pole except one, and has order $2$ at the exceptional pole, then it automatically takes the local expression \Ref{local form} around the exceptional pole. This point will be explained in Section 4, and we will also see what happens if  the exceptional pole is not a double pole of $\phi(z)dz^2$.

In the case that all poles except one have $d_c=2$, one can translate the conditions \Ref{local form}
into more explicit conditions on the coefficients of $\phi$. We will give several forms of such results in Section 5.
One particular case is (compare with \cite{E}):

\begin{theorem}\label{rational}\quad Fix $d\ge 2$. Let $c_1,\cdots, c_{2d-2}$  be distinct
points of $\C$, and $A_1,\cdots,A_{2d-2}$ be $2d-2$ unknown complex numbers. For $ i=1,\cdots,2d-2 $ and integers $m\ge 0$, set
 $$L_i= 3A_i^2-4\ds \sum_{j\ne i} \dfrac{A_j(c_i-c_j)+1}{(c_i-c_j)^2} \text{ and } E_{m+1}=\ds\sum_{i} ^{2d-2}\Big(mc_i^{m-1}+c_i^mA_i\Big).$$
Then the following $4$ problems have the identical set of solutions (for $A_i's$):
$$-\dfrac32 \sum_{i=1}^{2d-3}
\dfrac{A_i(z-c_i)+1}{(z-c_i)^2}\text { is equal to $S_f$ with $f$ a rational map of degree $d$},$$
$$\left\{\begin{array}{l}
L_i=0,\ i=1,\cdots,2d-2  \\
\ds E_{m+1}=0, \ m=0,1,2 \end{array}\right., \left\{\begin{array}{l}
L_i=0,\ i=1,\cdots,2d-3  \\
\ds E_{m+1}=0, \ m=0,1,2 \end{array}\right.,
\left\{\begin{array}{l}
L_i=0,\ i=1,\cdots,2d-2  \\
\ds E_{m+1}=0, \ m=0,1 \end{array}\right..$$

 In this case $f$ has simple critical points at $c_i$, $i=1,\cdots, 2d-2$ (and a regular point at $\infty$).
\end{theorem}

Notice that Eremenko in \cite{E} added one more equivalent problem to the above list, namely
$$\left\{\begin{array}{l}
L_i=0,\ i=1,\cdots,2d-2  \\
\ds E_2=0 \end{array}\right..$$

Since there are  more equations than unknowns, it is {\it  a priori} not clear that there are solutions
of such systems. Applying known results of L. Goldberg \cite{G} we could see that
the number of solutions is finite and non-zero, and depends on the set $c_1,\cdots, c_{2d-2}$.
Generically it is the Catalan number $u_d$, and otherwise it is at most $u_d$. See Section \ref{count} for a detailed account.

\REFTHM{merom}\quad  Fix $k\geq 1$. Let $c_1,\cdots, c_k$  be distinct
points of $\C$ and let $r_1,\cdots, r_k$ be complex numbers.
Define
 $$ e_i(z)=\prod_{j\neq i} \frac{z-c_j}{c_i-c_j} \ \ \mbox{and}  \ \
 v_i = \sum_{j\neq i}
    \left[
        -\dfrac32 \dfrac{1}{(c_i-c_j)^2} + \dfrac{r_j}{c_i-c_j} \right]
	.$$

Let $G$ be an entire function and let
 \begin{equation}
 \psi(z)=
 \sum_{i=1}^{k}
\left[ -\dfrac32 \dfrac{1}{(z-c_i)^2} + \dfrac{r_i}{z-c_i}
  + (-\frac{1}{2}r_i^2-v_i) e_i(z) \right]
+ G(z) \prod_{i=1}^{k} (z-c_i).
  \label{general merom}
  \end{equation}

Then  $\psi=S_f$ for a function  $f$
which is meromorphic in the entire complex plane and has
precisely $k$ critical points that are simple and
located at $c_1,\cdots,c_{k}$.

Conversely, any such function $f$ must have a Schwarzian
derivative of the form (\ref{general merom}).
\ENDTHM

Finally we are interested in the pole-dependency of the Schwarzian derivatives of rational maps of a given degree. We have obtained a complete answer on the cubic case:

\REFTHM{dependency}\quad
Let $\dfrac{z^3+ a_1z + a_0}{z^2+b_1z+b_0}$ be a rational function with 4 distinct critical points. Denote by $\De$ the cross ratio of these four points. Then the associated Schwarzian derivative depends locally holomorphically on the 4 critical points if and only if $a_1+3 b_0\ne 0$, and if and only if $\Delta\not\in \{ (1\pm i\sqrt{3})/2\}$.
\ENDTHM

\section{Local solution}

Let $U\subset\C$ be a domain. Recall that for a holomorphic map $f:
U\to \cbar$, its Schwarzian derivative $S_f$ can be expressed as $$S_f(z)=\left(\dfrac{f''}{f'}\right)'- \dfrac12 \left(\dfrac{f''}{f'}\right)^2 \ .$$
It is easy to check that $z_0\in U$ is a pole of $S_f(z)$ if and
only if $z_0$ is a critical point of $f(z)$. In that case the pole is
always a double pole. More precisely, if $\deg_{z_0}f=d>1$, then
$$
S_f=\dfrac{1-d^2}{2(z-z_0)^2} +\dfrac{a_1}{z-z_0}+a_2+\cdots
\quad \text{ as } z\to z_0.
$$

We address the inverse problem: Let $U\subset\C$ be a domain. Given
a meromorphic function $\phi(z)$ on $U$, find a holomorphic map $f:
U\to\cbar$ such that $S_f(z)=\phi(z)$.

The following result is well known:

\begin{lemma}[{see e.g. Theorem 1.1 in \cite{L}, pp. 53}]\label{well know}. If the domain $U$ is simply-connected and the meromorphic function
$\phi(z)$ has no poles in $U$,  then the equation $S_f(z)=\phi(z)$
has solutions. Moreover, if  $f(z)$ and $g(z)$ are two solutions,
 there is a M\"{o}bius transformation $\g$ of $\cbar$ such that
$g(z)=\g\circ f(z)$.
\end{lemma}

We consider first the case that $U$ is simply-connected,
$\phi(z)$ has only finitely many poles in $U$ and each pole of
$\phi(z)$ is a double pole. Due to the above statements, we only
need to consider the following local problem:

{\noindent \bf Inverse problem}. Let $d\ge 2$ be an integer and
$\phi(z)$ be a meromorphic function defined in a neighborhood of the
origin, with a double pole at $0$ and leading coefficient  $ \dfrac{1-d^2}{2}$, in other words, with the following local expression \REFEQN{c1}
\phi(z)=\dfrac{1-d^2}{2z^2}+\dfrac{a_1}{z}+a_2+\cdots\quad \text{ as
} z\to 0. \ENDEQN The question is: Under which conditions is there a holomorphic
function $f(z)$ defined in a neighborhood of the origin such that
$S_f(z)=\phi(z)$ ?

\begin{theorem}\label{local}\quad
Let $d\ge 1$ be an integer and $\phi(z)$ be a meromorphic function
defined in a neighborhood of the origin such that
$$ \phi(z)=\dfrac{1-d^2}{2z^2}+\dfrac{a_1}{z}+a_2+\cdots\quad \text{
as } z\to 0.$$ Then the equation $S_f=\phi$ has a meromorphic solution if and
only if \REFEQN{c2} \det\! \lv
a_1 & k_1 & 0 & \cdots & 0\\
a_2 & a_1 & k_2 & \cdots &0 \\
\vdots  & \vdots & \vdots & & \vdots\\
a_{d-1} & a_{d-2} & a_{d-3} & \cdots & k_{d-1}\\
a_d & a_{d-1} & a_{d-2} & \cdots & a_1 \rv = 0 \text{ where
$k_j=2j(j-d)$, $j=1,\cdots,d-1$.}\ENDEQN
\end{theorem}

\begin{proof} {\bf Step 1. Reduction to a linear differential equation}. We may look for solutions of the form
\REFEQN{f0} f(z)=\dfrac{z^d}d (1+b_1 z+ b_2 z^2 +\cdots)\text{ as } z\to 0.
\ENDEQN

Note that if such a solution $f$ exists the function  $f'(z)/z^{d-1}$ is holomorphic in a
neighborhood of the origin and takes the value 1 at the origin. Thus
there is a local holomorphic map $g(z)$ such that
$$
f'(z)=\dfrac{z^{d-1}}{g^2(z)},\text{ and } g(z)=1+c_1z+ c_2 z^2
+\cdots\text{ as }z\to 0.
$$
This motivates us to set $f(z)=\displaystyle{\int_{0}^z\frac{\zeta^{d-1}}{g^2(\zeta)}\textrm{d}\zeta}$, and express $S_f$ in terms of $g$.
By a direct computation, we have:
$$
\dfrac{f''}{f'}=\dfrac{d-1}{z}-\dfrac{2g'}{g}\text{ and }
\left(\dfrac{f''}{f'}\right)'=\dfrac{1-d}{z^2}-\dfrac{2(g''g-g'^2)}{g^2}.
$$
Thus
$$
S_f=\left(\dfrac{f''}{f'}\right)'- \dfrac12 \left(\dfrac{f''}{f'}\right)^2=\dfrac{1-d^2}{2z^2}+\dfrac{2(d-1)g'}{zg}-\dfrac{2g''}{g}.
$$

On the other hand, set
\REFEQN{q}q(z)=z\left(\phi(z)-\dfrac{1-d^2}{2z^2}\right)=a_1+a_2z+\cdots.\ENDEQN
Then the equation $S_f(z)=\phi(z)$ takes the following form:
\REFEQN{linear-d} 2(d-1)g'-2zg''=qg.\ENDEQN This is a linear
differential equation.

{\bf Step 2. Solving the differential equation.}

For later purpose, we solve the equation (\ref{linear-d}) with $d$ replaced by an arbitrary complex number $\de$.

\REFLEM{Linear general}\quad Let $q(z)=a_1+a_2z+\cdots$ be a convergent power series in a disk $\{|z|< R\}$ and $\de$ an arbitrary complex number. Then the equation
\REFEQN{linear} 2(\de-1)g'-2zg''=qg\ENDEQN
has a formal power series solution in the form
\REFEQN{series} g(z)=1+c_1z+ c_2 z^2
+\cdots\text{ as }z\to 0
\ENDEQN
if and only if: either $\de$ is not a strictly positive integer, or $\de=d\ge 1$ is an integer and the coefficient $a_d$ of  $q(z)$ is related to the previous coefficients by the condition \Ref{c2}.
In both cases the power series of $g$ converges in the disk $\{|z|< R\}$.
\ENDLEM
\begin{proof}
Set $c_0=1$. The
right hand side of \Ref{linear} is equal to
$$
qg=\ds\sum^{\infty}_{n=1}\left(\ds\sum_{j= 0}^{n-1}a_{n-j}c_j\right)z^{n-1},
$$
whereas the left hand side takes the expansion, by setting $k_n=2n(n-\de)$,
$$
2(\de-1)g'-2zg''=\ds\sum^{\infty}_{n=1}[2n(\de-n)]c_nz^{n-1}=
\ds\sum^{\infty}_{n=1}-k_nc_nz^{n-1}.
$$
Comparing the coefficients, we get the recursive formula:
\REFEQN{coefficient}\forall\, n\ge 1,\quad  -k_nc_n=\ds\sum_{j=0}^{n-1}a_{n-j}c_j=
a_n+a_{n-1}c_1+\cdots a_1c_{n-1}. \ENDEQN
One may rewrite it in matrix form:
\REFEQN{all} \forall\, n\ge 1,\quad \lv a_n & a_{n-1} & \cdots & a_1 & k_n & 0 & \cdots \rv \lv  1 \\ c_1 \\ \vdots \\ c_{n-1} \\ c_n\\ c_{n+1}  \\ \vdots \rv=0.\ENDEQN

As $k_n=2n(n-\de )$, we have $k_n=0$ if and only if $\de $ is a strictly positive integer and $n=\de$.

{\noindent \bf Case 1}. Assume that $\de$ is a complex number that is not a strictly positive integer.
Then $k_n\ne 0$ for all $n\ge 1$. So, using $c_0=1\ne 0$, the recursive formula \Ref{coefficient} determines a unique
sequence $c_n$, $n\ge 0$ so that its generating function \Ref{series} is a formal power series solution of of linear equation \Ref{linear}.

{\noindent\bf Case 2}. Assume that $\de$ is a strictly positive integer. Set $\de=d$. In this case $k_d=0$ and $k_n\ne 0$ for any $n\ne d$.
It is then easy to see that the following statements are equivalent:
\begin{enumerate}
\item  the equation \Ref{linear} has a power series solution in the form of \Ref{series}
\item the following linear system has a solution with $c_0=1$,
$$\lv
a_1 & k_1 & 0 & \cdots & 0\\
a_2 & a_1 & k_2 & \cdots &0 \\
\vdots  & \vdots & \vdots & & \vdots\\
a_{d-1} & a_{d-2} & a_{d-3} & \cdots & k_{d-1}\\
a_d & a_{d-1} & a_{d-2} & \cdots & a_1 \rv \lv c_0 \\ c_1 \\ \vdots \\
c_{d-2} \\ c_{d-1} \rv = 0,$$
\item the same system has a non-zero vector solution, and
\item  the determinant of the square matrix vanishes, that is, the condition \Ref{c2}.
\end{enumerate}

{\noindent\bf Convergence of the formal solutions in all cases.} Let $R>0$ be the convergence radius of $q(z)=\sum_{n\ge 1} a_n z^{n-1}$. Let $\de$ be any complex number.  Let $(c_0,c_1,\cdots)$ be a solution of the system \Ref{coefficient}.
 We want to show that   for $|z|<R$ the power series $\ds \sum_{m=0}^\infty c_m z^m$
 converges.

 Fix $0<\rho<R$, and $0<S<+\infty$ such  that
 $$\sum_{m=1}^\infty |a_m| \rho^{m-1}\le  S.$$

 Now fix an integer $m_0$ such that for $m>m_0$, we have $|k_m|=2m|m-\de|\ge  2S\rho$. By \Ref{coefficient},
 for $m>m_0$,
  $$|c_m|\rho^m \le \dfrac {|a_m|\rho^m (|c_{ 0}|\rho^0)+\cdots+ |a_1|\rho^1 ( |c_{m-1}| \rho^{m-1})}{|k_m|}\le \dfrac {S\rho}{|k_m|} \cdot {\ds \max_{l<m} {|c_l|\rho^l}}\le\dfrac 12 \max_{l<m} {|c_l|\rho^l}.$$

 It follows that the sequence $|c_m|\rho^m$ is bounded. Therefore $\sum_{m=0}^\infty c_m z^m$ converges for any $|z|<\rho < R$.
 As $\rho<R$ was arbitrary, we conclude that the series $\ds \sum_{m=0}^\infty c_m z^m$ converges on the disk of convergence of the power series of $q(z)$.\end{proof}

{\noindent\bf Step 3. Conclusion.} Let now consider to the case that $\de=d\ge 1$ is an integer.

 If the coefficients of $q(z)$ do not satisfy \Ref{c2}, then the equation \Ref{linear} does not have a solution $g(z)$ in the form $1+c_1 z+\cdots $. So the equation $S_f=\phi$ does not have a solution $f$ normalized so that
$f(z)=\dfrac{z^d}d(1+O(z))$. Suppose that $f_1$ is a meromorphic solution of $S_f=\phi$ in a neighborhood of the origin. Then there is a  M\"{o}bius transformation $\g$ of $\cbar$ such that $\widetilde{f_1}:=\g\circ f_1$ has the form $\dfrac{z^m}m(1+O(z))$, where $m\geq1$ denotes the local degree of $f_1$ at the origin. Note that the holomorphic function $\widetilde{f_1}$ is also a solution of $S_f=\phi$ by \eqref{composition-formula}. We then get a contradiction.

Assume now that the coefficients of $q(z)$ do satisfy \Ref{c2}. Then the equation \Ref{linear} has a solution $g(z)$ in the form $1+c_1 z+\cdots $, whose convergence disk contains the convergence disk of $q(z)$.
  Thus, the holomorphic function $f(z)=\displaystyle{\int_0^z\frac{\zeta^{d-1}}{g^2(\zeta)}\textrm{d}\zeta}$ is a Schwarzian primitive of $\phi$ in the convergence disk the power series of $q(z)=z\left(\phi(z)-\dfrac{1-d^2}{2z^2}\right)$.
\end{proof}

Notice that in the case that the leading coefficient of $\phi$ is $\dfrac{1-d^2}{2}$, one may choose $\de=-d$ and obtain a  solution of \Ref{linear} in the form $g(z)=1+c_1 z+ \cdots$, independent of the determinant in \Ref{c2} being zero or not. There is therefore $r>0$ a constant so that $g(z)\ne 0$ for every point $z$ with $|z|<r$.
 Fix $0<r_0< r$.  Set $z_0=r_0$ and
 \REFEQN{gg}\widehat{f}(z)=\int_{z_0}^z\dfrac{\zeta^{\de-1}}{g^2(\zeta)}\textrm{d}\zeta= \int_{z_0}^z\dfrac{g^{-2}(\zeta)}{\zeta^{d+1}}\textrm{d}\zeta,\ENDEQN
 where $z$ belongs to a disk neighborhood of $z_0$ disjoint from $0$.
 Denote by $\widehat{b}_d$ the $\zeta^d$ coefficient of the power series expansion of
 $ g^{-2}(\zeta)$. Let $\log z$ be the branch of logarithm in a neighborhood of $z_0$ so that $\log z_0\in \R$.
Then $$\widehat{f}(z)=\dfrac{1}{-d\cdot z^d(1+H(z))} + \widehat{b}_d(\log z -\log z_0)$$ for some holomorphic function $H$ defined in a neighborhood of $z_0$, and $\widehat{f}$ is locally a \sp of $\phi$ around $z_0$.

 This $\widehat{f}$ has an analytic continuation and when making a full turn around $0$ its value
 has a phase shift of $\widehat{b}_d\cdot 2\pi i$, with $H$ extending to a holomorphic function in a neighborhood of $0$ and $H(0)=0$.

We have proved:

 \begin{corollary}[{second criterion}]\label{second}
 Assume that $\phi$ is a meromorphic function in a neighborhood of the origin with the local expression
 $$ \phi(z)=\dfrac{1-d^2}{2z^2}+\dfrac{a_1}{z}+a_2+\cdots\quad \text{
as } z\to 0$$
for some integer $d\ge 1$.
Then the equation $S_f=\phi$ admits  always an eventually multivalued solution $\widehat{f}$ in the form
 $$\widehat{f}(z)=\dfrac1{-d\cdot z^d(1+ H(z))} + \widehat{b}_d \log z\ ,$$
 where $H$ is a holomorphic map (with value in $\C$) around $0$ with $H(0)=0$.
 And $S_f=\phi$ admits a holomorphic solution in a neighborhood of $0$  if and only if $\widehat{b}_d=0$.
\end{corollary}

 This criterion has the advantage that it tells the form of an eventual obstruction (the log term, see Theorem \ref{developing} below for an application), but has
 the disadvantage that the condition $\widehat{b}_d=0$ can not be explicitly expressed as the determinant in \Ref{c2}.

 \begin{remark}\quad
 It is known by a direct calculation that if $S_f=\phi$, then $v=\dfrac1{\sqrt{f'}}$ satisfies the Sturm-Liouville equation
   \REFEQN{fro}z^2 v'' + \dfrac{z^2\phi(z)}2 v=0 .\ENDEQN
 Looking for solutions in the form
 \REFEQN{half} v(z)=z^{\frac{1-d}2}\sum_{k=0}^\infty v_k z^k \ENDEQN
 will lead to the
 same criteria. This approach was used by Eremenko \cite{E}, with a somewhat different presentation.
 For instance only the case $d=2$ was worked out explicitly there.
\end{remark}
\section{Developing maps and holonomy around a puncture}

In this section, we give a geometric interpretation of the conditions
(\ref{c1}) and (\ref{c2}).

We may consider the Schwarzian equation on any multiply-connected
domain $U\subset\cbar$. Let $\phi(z)$ be a holomorphic function on
$U$. Then $S_f=\phi$ has a solution on any simply-connected sub-domain in
$U$. Start from a fixed disk $D$ in $U$, we have a solution $f_0$ in
$D$ which is unique up to a M\"{o}bius transformation. The function element $(f_0,D)$ admits an analytic continuation along any paths in $U$ starting from a point of $D$. For any path
$p: [0,1]\to U$ with $p(0)=p(1)\in D$, by following an analytic continuation of $(f_0,D)$ along
$p$, we get another solution $f_p$ on $D$ and thus there is a
M\"{o}bius transformation $\g_p$ such that $f_p=\g_p\circ f_0$. Note
that $\g_p$ is determined by the homotopy class of $p$ and the
choice of $f_0$. Thus $\g_p$ is unique up to a M\"obius
transformation conjugacy.

Denote by ${\mathcal M}$ the group of M\"obius transformations of
$\cbar$. Then we have defined a group homomorphism from the
fundamental group $\pi_1(U)$ to a sub-group $G_{\phi}$ in ${\mathcal M}$ up to a M\"obius transformation conjugacy. The group $G_{\phi}$
is called the {\bf holonomy group}\/ of the
complex projective structure on $U$ induced by $\phi$. The equation $S_f=\phi$
has a global solution if and only if $G_{\phi}$ is trivial (i.e. it contains only the identity).

When $U$ is an annulus, for example a punctured disk,  the fundamental group of $U$ is
isomorphic to $\Z$. Therefore, $G_{\phi}$ is generated by a single
transformation $\g_{\phi}\in{\mathcal M}$. Either $\g_{\phi}$ is
hyperbolic or elliptic and its conjugate class is determined by its
eigenvalue at fixed points, or $\g_{\phi}$ is conjugated with the
parabolic transformation $z\to z+1$, or $\g_{\phi}$ is the identity.

\begin{theorem}\label{developing}\quad Assume that $\phi$ is a meromorphic function in a neighborhood of the origin  with the local expression
$$ \phi(z)=\dfrac{1-\delta^2}{2z^2}+\dfrac{a_1}{z}+a_2+\cdots\quad \text{
as } z\to 0,\quad \delta\in\C.$$
Then, for $U$ a punctured neighborhood of $0$,
\begin{description}
\item [(a)] if $\delta$ is not an integer, the generator $\g_\phi$ is a non-identity  elliptic transformation;
\item [(b)] if $\delta=0$, the generator $\g_\phi$ is a parabolic transformation;
\item [(c)] if $|\delta|$ is an integer greater than or equal to $1$, the generator $\g_{\phi}$ is a parabolic transformation or the identity, and furthermore,
  the generator is equal to the identity if and only if $\phi(z)$ satisfies additionally the condition (\ref{c2}).
  \end{description}
\end{theorem}

\begin{proof} For simplicity we assume that
  $\phi(z)$ is a holomorphic map on the punctured disk $ \D^*$ with values in $\C$ (in particular $\phi(z)$ has no poles on $\D^*$). In this case the power series $q(z)=z(\phi(z)-\dfrac{1-\delta^2}{2z^2})=a_1+a_2z+\cdots$  converges in $\D$.

For a non-integer complex number $\delta$, we define $z^\de=e^{\de \log z}$ for some determination of $\log z$.

{\bf I.} Assume that either $\delta$ is not an integer or $\delta=0$. We see from Lemma \ref{Linear general} that equation (\ref{linear}) always has a holomorphic solution $g(z)=1+c_1z+\cdots$ in $\D$. Therefore, we get a Schwarzian primitive $\widetilde{f}$ of $\phi$ in a small disk $D$ close but disjoint from the origin such that
$$\tilde f(z)=\left\{\begin{array}{ll}\ds\int_{z_0}^z\dfrac{\zeta^{\de-1}}{g^2(\zeta)}d\zeta=\dfrac{z^\de}{\de}(1+h(z))+C_0 & \text{ if  $\de$ is not an integer}\vspace{0.2cm}\\
\ds\int_{z_0}^z\frac{1}{\zeta g^2(\zeta)}d\zeta=\log z+v(z)+C_1 & \text{ if  $\de=0$}\end{array}\right.,
$$
 where $h,v$ are holomorphic maps near $0$ with $h(0)=0=v(0)$ and $C_0,C_1$ are constant.
 Post composing $\tilde f$ by a translation to get rid of the constants, we get a \sp of the form
 $$ f(z)=\left\{\begin{array}{ll}\ds\int_{z_0}^z\dfrac{\zeta^{\de-1}}{g^2(\zeta)}d\zeta=\dfrac{z^\de}{\de}(1+h(z))& \text{ if  $\de$ is not an integer}\vspace{0.2cm}\\
\ds\int_{z_0}^z\frac{1}{\zeta g^2(\zeta)}d\zeta=\log z+v(z)& \text{ if  $\de=0$}\end{array}\right.,
$$
 where $h,v$ are holomorphic maps near $0$ with $h(0)=0=v(0)$, and $z^\de=e^{\de \log z}$ for some fixed determination of $\log z$.

An analytic continuation of $f(z)$ along the curve  $[0,1]\ni \alpha\to e^{2\pi i \al}z_0$  for some $z_0\in D$ gives another Schwarzian primitive $f_1$ of $\phi$ on $D$ such that
\[f_1(z)=\left\{\begin{array}{llll}\ds e^{2\pi i\de}\dfrac{z^\de}{\de}(1+h(z))= \g_\phi\circ f(z), & \text{  $\g_\phi(w)=e^{2\pi i\de }w$} & \text{ if  $\de$ is not an integer}\vspace{0.2cm}  \\
 \log z+2\pi i+v(z)=\g_\phi\circ f(z), & \text{ $\g_\phi(w) =w+2\pi i$} & \text{ if $\de =0$} \end{array}\right..\]

{\bf II.} Assume now that $d:=|\delta|\geq1$ is an integer.  For $D$ a small disk close but disjoint to $0$, the holomorphic map
 $$\widehat{f}(z)= \dfrac{z}{-d\cdot (1+ H(z))} + \widehat{b}_d\log z:D\to \widehat{\C}$$
 given by Corollary \ref{second} is a solution of the equation $S_f=\phi$ on $D$. By following an analytic continuation along the same curve as above we get another solution on $D$
 \[\widehat{f}_1(z)= \dfrac{z}{-d\cdot (1+ H(z))} + \widehat{b}_d\log z+\widehat{b_d}\cdot 2\pi i=\g_\phi\circ \widehat{f}(z),\ \g_\phi(w)=w+2\pi i \widehat{b}_d.\] The map $\g_\phi$ is a M\"{o}bius map having a parabolic fixed point at $\infty$.  It is the identity map if and only if  $\widehat{b}_d=0$. \end{proof}

Applying analytic extension, we have:

\begin{corollary}\label{domain}\quad Let $U\subset\C$ be a simply-connected domain
and $\phi(z)$ be a meromorphic function on $U$ with only finitely
many poles in $U$. Suppose that at each pole $z_0\in U$, the map $\phi$ takes a local expansion
of the form
$$
\phi(z)=\dfrac{1-d^2}{2(z-z_0)^2}+\dfrac{a_1}{z-z_0}+
a_2+\cdots,\text{ as }z\to z_0,\quad d\in \N, d\ge 2, a_i\in \C.
$$
Then the holonomy group of $\phi$  is generated by finitely many parabolic M\"{o}bius transformations.  Furthermore, there is a holomorphic map
$f:\ U\to\cbar$ such that $S_f(z)=\phi(z)$ if and only if $\phi$ satisfies the condition \Ref{c2} at each pole, and if and only if the holonomy group is reduced to the identity. In this case $f$ is unique
up to post composition by M\"{o}bius transformations. \end{corollary}

\section{Schwarzian derivative near singular points}

Following Corollary \ref{domain}, we know that if $\phi(z)$ is a meromorphic function on $\C$ with only finitely
many poles, and  takes a local expansion \Ref{local form} around each pole, then there exists a holomorphic map $f:\C\to \widehat{\C}$ fulfilling $S_f=\phi$.
Thus, to prove Theorem \ref{except one}, we need to study the behavior of  $S_f$ near $\infty$.

\begin{theorem}\label{simple-pole}\quad Set $\D^*=\{z:\ 0<|z|<1\}$. Let $f: \D^*\to
\cbar$ be a locally injective holomorphic map and $S_f(z)$ be its
Schwarzian derivative. Then we have the following: \\
(i) The origin is never a simple pole of $S_f(z)$. \\
(ii) If the origin is a double pole of $S_f(z)$,  then $f$ extends to a holomorphic function on $\D$ with a critical point at $0$. In this case $S_f$ has a double pole at $0$ with leading coefficient $\dfrac{1-d^2}2$ for some integer $d\ge 2$, and $S_f$ satisfies the condition \Ref{c2}.
\end{theorem}

\begin{proof}
Set $\phi(z)=S_f(z)$. We just consider the case that the origin is either a double pole, a simple pole or a regular point of $\phi$. There are therefore complex numbers $\de$, $a_n$ ($n\ge 1$) such that
$$\phi(z)=\dfrac{1-\de^2}{2z^2} + \dfrac {a_1}z+a_2+\cdots .$$
 We may choose $\de$ so that $\Re\de\ge 0$.

 If either $\de$ is not an integer or $\delta=0$, we know from  Theorem \ref{developing} (a), (b) that $\phi$ can not have a \sp in a punctured neighborhood of $0$.

If $\de=d\geq 1$ is an integer and the condition \Ref{c2} is  not satisfied,  by Corollary \ref{second} and
  Theorem \ref{developing} (c), the meromorphic function $\phi$ can not have a \sp in a punctured neighborhood of $0$.

If $\de=d=1$, note that  the condition \Ref{c2} is equivalent to $a_1= 0$, so that $\phi$ has no pole at $0$.

Therefore, in order to have $\phi=S_f$ in a punctured neighborhood of $0$ we must have
either $\de=1$ and $\phi$ has no pole at $0$, or $\de=d$ is an integer greater than or equal to 2 and
$\phi$ satisfies the condition \Ref{c2}. In both cases, the map $f$ has an analytic continuation on $\D$ according to the discussion in section 2.
\end{proof}

Note that the above proof gives the nature of Schwarzian primitives for a meromorphic map $\phi$
in a  neighborhood of $0$ with at worse double pole at $0$.


\begin{proof}[\bf Proof of Theorem \ref{except one}]
Point (a) follows directly from Theorem \ref{local}. To prove point (b), without loss of generality, we assume that the exceptional pole of $\phi$ is $\infty$.
By Corollary \ref{domain}, we obtain a Schwarzian primitive $f$ of $\phi$ on $\C$. Since $\infty$ is a double pole of $\phi$, by Theorem \ref{simple-pole} (ii) $f$ extends to a rational map and $S_f=\phi$ takes the local expression (1) around $\infty$ after change of coordinates by $1/z$.
\end{proof}

 \section{Schwarzian derivatives of rational maps, a global study}

In this section, we give a necessary and sufficient condition for a meromorphic
quadratic differential to be the Schwarzian derivative of  a polynomial, or a rational
map, in the case that all except one poles have $d_c=2$.

 Let $c_1,\cdots, c_{k}$ ($k\ge 1$) be distinct
points of $\C$. Set
\begin{equation}\label{phi}
\phi(z):=-\dfrac32 \sum_{i=1}^{k}
\dfrac{A_i(z-c_i)+1}{(z-c_i)^2},
\end{equation}
and
$$L_i:=  3A_i^2-4\ds \sum_{j\ne i} \dfrac{A_j(c_i-c_j)+1}{(c_i-c_j)^2},\ (i=1,\cdots,k).$$
Recall that if a meromorphic function has a double pole at $c$, the number $d_c$ is defined by the expression
\[\frac{1-d_c^2}{2(z-c)^2}+\dfrac{a_1}{z-c}+a_2+\cdots\quad \text{as $z\to c$}.\]

\begin{lemma}\label{special}\quad
For each $c_i$,  the condition \Ref{c2} holds at $c_i$ for $\phi$ of form (\ref{phi}) if and only if $L_i=0$.
\end{lemma}
\begin{proof} Fix a pole $c_i$.
The map $\phi(z)$ has the local expression
\[\phi(z)=\dfrac{1-2^2}{2(z-c_i)^2}+\dfrac{-\frac{3}{2}A_i}{z-c_i}-\dfrac{3}{2}\sum_{j\not=i}\dfrac{A_j(z-c_j)+1}{(z-c_j)^2}.\]
As a double pole,   we have $d_{c_i}=2$, $a_1=-3A_i/2$ and $a_2=-\dfrac{3}{2}\sum_{j\not=i}\dfrac{A_j(c_i-c_j)+1}{(c_i-c_j)^2}.$
In this case, condition \Ref{c2} holds if and only if $a_1^2-k_1a_2=0$, if and only if $L_i=0$.
\end{proof}

In order to study the behavior at $\infty$ we set $$E_{m+1}= \sum_{i=1}^k (mc_i^{m-1} + c_i^m A_i),\ m\ge 0.$$ Then
\begin{eqnarray}
   \phi(z) & = & -\dfrac32
      \sum_i
	\left(
	   \frac{1}{z^2} \left( 1 - c_i \frac{1}{z} \right)^{-2}
           +
           A_i \frac{1}{z}\left(1-c_i \frac{1}{z} \right)^{-1}
	 \right)
	\\
    & = & -\dfrac32 \left(
      \frac{1}{z}  E_1 + \frac{1}{z^2} E_2 + \frac{1}{z^3} E_{3}+  \frac{1}{z^4} E_{4} +\cdots \right).
\end{eqnarray}
Changing coordinates $w=1/z$,
the quadratic differential becomes:
$$\phi(z)dz^2 = \phi(w^{-1}) w^{-4} dw^2 =: \Phi(w) dw^2$$
with
$$ \Phi(w) = -\dfrac32
  \left( \dfrac{E_{1}}{w^3} + \dfrac{E_{2}}{w^2} + \dfrac{E_{3}}{w^1} \right)
    + O(w^0).$$
Clearly,
\REFLEM{E}\begin{eqnarray} E_1\ne 0 &\Longleftrightarrow& \infty \text{ is a triple pole of } \phi(z)dz^2\\
E_1=0\ne E_2&\Longleftrightarrow& \infty \text{ is a double pole of } \phi(z)dz^2 \text{ with leading coef. } \dfrac{3E_2}{-2}\\
E_1=E_2=0\ne E_3&\Longleftrightarrow& \infty \text{ is a simple pole of } \phi(z)dz^2\\
E_1=E_2=E_3=0 &\Longleftrightarrow& \infty \text{ is a regular point of } \phi(z)dz^2\ .\end{eqnarray}
\ENDLEM

Note that for $\dfrac{W\cdot (z^m+l.o.t.)}{z^{m+2}+l.o.t} dz^2$, with $W\ne 0$, the point $\infty$ is a double pole
  and after change of coordinates, one may check that the leading coefficient of the quadratic differential at $\infty$ is $W$.

\REFTHM{polynomial}\quad
(1) If $\phi(z)=S_P(z)$ for $P$ a polynomial with $P'(z)=\prod_i (z-c_i)$,
then $$A_i=\dfrac23\sum_{j\ne i} \dfrac1{c_i-c_j},\ i=1,\cdots,k \ .$$
They verify:
$L_i=E_1=0$ for $1\le i\le k$, and $-\dfrac32 E_2=\dfrac{1-(k+1)^2}2$.

(2) Conversely, if $L_i=E_1=0$ for $1\le i\le k$, and $-\dfrac32 E_2=\dfrac{1-(k+1)^2}2$, then
$\phi(z)=S_P(z)$ with $P'(z)=\ds\prod_i (z-c_i)$.
\ENDTHM

\begin{proof}
(1) The formula for $A_i$ can be checked using the formula of $S_P$. The fact $L_i=0$ for $1\le i\le k$
is due to Theorem
\ref{local} and Lemma \ref{special}.

Note that each $c_i$ is a simple critical point of $P$ and $\infty$
 is  a critical point of $P$ of local degree $k+1$.
 It follows that $\infty$ is a double pole of $S_P(z)dz^2$ with leading coefficient $\dfrac{1-(k+1)^2}2$.

  It follows in our case that $E_1=0$ and $-\dfrac 32 E_2=\dfrac{1-(k+1)^2}2$.

(2) By Lemma \ref{special} and Theorem \ref{local}, if $L_i=0$ for $1\le i\le k$, then
there is a holomorphic map $f: \C\to\cbar$ such that
$S_f(z)=\phi(z)$. Moreover,   $E_1=0$ and  $-\dfrac 32 E_2=\dfrac{1-(k+1)^2}2$
implies that  $\phi(z)dz^2$ has a double pole at $\infty$. By Theorem \ref{simple-pole} our map $f$ entends to a holomorphic map at the puncture $\infty$, so $f$ is a rational map.
Furthermore, $\infty$ is a critical point of local degree $k+1$. So $f$ is a polynomial.
\end{proof}

The following result is to be compared with results in \cite{E}.

\REFTHM{Rational}\quad (a) Assume $k=2d-2$. If there is a rational map $f$ with $\deg(f)=d$ such that
$S_f(z)=\phi(z)$, then $L_i=E_j=0$ for $1\le i\le 2d-2$ and $j=1,2,3$. Conversely\begin{enumerate} \item  if $L_i=E_j=0$
for $1\le i\le 2d-3$ and $j=1,2,3$ then automatically $L_{2d-2}=0$ and there is a rational map $f$ with $\deg(f)=d$ such that
$S_f(z)=\phi(z)$. \item Or  if $L_i=E_j=0$
for $1\le i\le 2d-2$ and $j=1,2$ then automatically $E_3=0$ and there is a rational map $f$ with $\deg(f)=d$ such that
$S_f(z)=\phi(z)$.\end{enumerate}

(b) If $L_i=0$ for $1\le i\le k$, then there is a
holomorphic map $f: \C\to\cbar$ such that $S_f(z)=\phi(z)$.
Moreover,
\begin{enumerate}\item if
$E_1\ne 0$, then $f$ has an essential singularity at $\infty$.
\item if $E_1=0$ and $E_2\ne 0$, the $f$ is a rational map   with critical set $\{c_1,\cdots, c_{k},\infty\}$.
And $-\dfrac32 E_2=\dfrac{1-(m+1)^2}2$ for some integer $1\le m\le k$. The map $f$ has local degree $m+1$ at $\infty$.
In particular $k+m$ is an even number and is equal to $2\deg(f)-2$.
\item  if $E_1=E_2=0$, then automatically $E_3=0$ and $f$ is a rational map
with $\deg(f)=d=\dfrac{k+2}2$, with critical set $\{c_1,\cdots, c_k=c_{2d-2}\}$.
\end{enumerate}

(c)  If $k=2d-3$, $L_i=0$ for $1\le i\le 2d-3$, $E_1=0$ and $E_2\ne 0$, Then  $ E_2=1$ and  $\phi=S_f$ for a rational map $f$ with simple critical points at $c_i$, $i=1,\cdots, 2d-3$ and a simple critical point at $\infty$. In particular $f$ is of degree $d$.
 \ENDTHM

\begin{proof}
(a) Let $k=2d-2$. If there is a rational map $f$ with $\deg(f)=d$ such that
$S_f(z)=\phi(z)$, then $L_i=0$ for $1\le i\le 2d-2$ by Theorem
\ref{local} and Lemma \ref{special}.

Note that each $c_i$ is a simple critical point of $f$ and thus the
infinity is not a critical point of $f$ (recall that $f$ has $2d-2$ critical points, counted with multiplicity). Thus the infinity is a
regular point of the quadratic differential $\phi(z)dz^2$.   Therefore $E_i=0$ for $i=1,2,3$.

Conversely, for Case 1, one may assume that the exceptional pole $c_{2d-2}$ is located at $0$.
By Lemma \ref{special} and the conditions $L_i=0$, $i=1,\cdots, 2d-3$, the condition \Ref{c2} is
satisfied at $c_i$, $i=1,\cdots, 2d-3$. By Lemma \ref{E}, the conditions $E_j=0$, $j=1,2,3$ tell that
$\infty$ is a regular point of the quadratic differential $\phi(z)dz^2$.
Thus there is a \sp $f$ of $\phi(z)dz^2$ defined on the simply connected domain $\cbar\smm \{0\}$.
Now we may apply Theorem  \ref{simple-pole} to show that $f$ extends meromorphically to $0$.
So $f$ is a rational map of degree $d$.

Case 2 will be covered by the case (b).3 below.

(b) By Lemma \ref{special} and Theorem \ref{local}, if $L_i=0$ for $1\le i\le k$, then
there is a holomorphic map $f: \C\to\cbar$ such that
$S_f(z)=\phi(z)$ and each $c_i$ is a simple critical point of $f$.

1. If $E_1\ne 0$, then $f$ can not be extended to a rational map at $\infty$. So $\infty$ must be an essential singularity.

2. If $E_1=0$ and $E_2\ne 0$. Then $\infty$ is a double pole of $\phi(z) dz^2$. By Theorem  \ref{simple-pole}
there is an integer $d_\infty\ge 2$ and $f$ extends to a rational map with local degree $d_\infty$ at $\infty$.

Note that at most half of the critical points of a rational map can be merged into a single one. So $d_\infty-1\le k$.

3. If  $E_1=E_2=0$, then the infinity is
either a regular point or a simple pole of the quadratic
differential $\phi(z)dz^2$. According to Theorem \ref{simple-pole} (i), the latter case never happens, so  necessarily $E_3=0$. Since $\infty$ is a regular point of $\phi(z)dz^2$, then $f$ extends to a rational map with critical points $c_1,\ldots,c_k$.

(c) is a particular case of (b).2 with $m=1$. It follows that $E_2=1$.
\end{proof}

\begin{example}
Let us choose two critical points $c_1=1$, $c_2=0$. We get
$(A_{c_1}, A_{c_2})=(3,-3)$ or $(-1,1)$. In the first case we get
$$\phi_1(z)=-\dfrac{3}{2z^2(z-1)^2}, \quad
f_1(z)=\dfrac{z^2}{(z-1)^2}.$$
In the second case,
$$\phi_2(z)=-\dfrac{8z^2-8z+3}{2z^2(z-1)^2},\quad
f_2(z)=2z^3-3z^2.$$
\end{example}

\begin{proof}[\bf Proof of Theorem \ref{rational}]
It follows directly from Theorem \ref{Rational} (a). \end{proof}

\begin{proof}[\bf Proof of Theorem \ref{merom}] Note that a meromorphism $\psi$ on $\C$ has a global expression  (\ref{general merom}) if and only if its poles are exactly $c_1,\ldots,c_k$ and it takes the form
\[\psi(z)=\dfrac{1-2^2}{2(z-c_i)^2}+\dfrac{r_i}{z-c_i}-\frac{r_i^2}{2}+\cdots, \text{ as $z\to c_i$}\]
around each $c_i$.  It follows that each $d_{c_i}=2$ and $\psi$ satisfies condition (\ref{c2}) at each pole. Then, by Corollary \ref{domain}, we get $\psi=S_f$ for a holomorphic map   $f:\C\to \widehat{\C}$ which has
precisely $k$ critical points that are simple and
located at $c_1,\cdots,c_{k}$, and vice versa.\end{proof}

\section{Parameter dependence}

Given a set $\CCC$ of $2d-2$ district points in $\cbar$, we know that there are finitely many choices of
quadratic differentials $\phi_i$ with double poles at $\CCC$ that are Schwarzian derivatives of degree-$d$ rational maps.
One may ask the question whether each $\phi_i$ depends holomorphically on the set $\CCC$.

The answer is no. Here is a counter example, using a family studied by \cite{B}. Set $j:=e^{2\pi i/3}$ and let
$$h_{\alpha}(z)=\dfrac{{\alpha}(z^3+2)+3z^2}{3{\alpha}z+2z^3+1},\quad \begin{array}{c|cccc}\text{critical\ points} & 1&j& j^2&{\alpha}^2\\  \hline \text{critical \ values}& 1&j^2&j &\dfrac{{\alpha}^4+2{\alpha}}{2{\alpha}^3+1}\end{array}.$$

When ${\alpha}^6\ne 1$, we have a degree 3 rational map with four simple critical points.

One can check easily that if $M\circ h_{\alpha}=h_\beta$ for some pair ${\alpha},\beta$ and a M\"obius map $M$, then ${\alpha}=\beta$ and $M=id$. Thus the Schwarzian derivative map ${\alpha} \mapsto S_{h_{\alpha} }$ is injective. Now
$$S_{h_{\alpha} }(z)=-\dfrac32\cdot  \dfrac{1 + 4 {\alpha} ^3 - 4 {\alpha}  z + 8 {\alpha} ^4 z - 18 {\alpha} ^2 z^2 + 8 z^3 - 4 {\alpha} ^3 z^3 +
    4 {\alpha}  z^4 + {\alpha} ^4 z^4}{ (z-{\alpha} ^2)^2 (z^3-1)^2}$$
    $$=-\dfrac3{2 (z-1)^2} + \dfrac{-2 + 2 {\alpha}  + {\alpha} ^2}{({\alpha} ^2-1) (z-1)} - \dfrac 3{
 2 (z-{\alpha} ^2)^2} + \dfrac {3 (-2 {\alpha}  + {\alpha} ^4)}{({\alpha} ^6-1) (z-{\alpha} ^2)} + $$
 $$ +\dfrac 9{
 2 (1 + z + z^2)^2}+ \dfrac{-1 - 2 {\alpha}  - 3 {\alpha} ^2 - 4 {\alpha} ^3 - 2 {\alpha} ^4 - 2 z -
  4 {\alpha}  z - 3 {\alpha} ^2 z - 2 {\alpha} ^3 z - {\alpha} ^4 z}{(1 + {\alpha} ^2 + {\alpha} ^4) (1 + z + z^2)}$$

    Thus ${\alpha} \mapsto S_{h_{\alpha} }$ holomorphically embeds $\C\smm \{\pm1,\pm j,\pm j^2\}$ into the space
    of degree-8 rational maps.

    Fixing a critical set $\CCC=\{1,j,j^2, c={\alpha} ^2\}$, there are generically two ${\alpha} $'s and thus two Schwarzians realizing this set, except when ${\alpha} =0$.

The dependence of each branch $S_{h_{\alpha} }$,  and thus of the parameter ${\alpha} $, on the critical set $\CCC=\{1,j,j^2, c\}$ is locally holomorphic
except when $c=0$.

From this example one may guess that generically each branch of the Schwarzian depends holomorphically on the critical set. We have made this idea precise in the following setting.

Fix an integer $\mu \ge 1$, and set $d=\mu+1$. Choose a vector $\aa\in \C^{2\mu}$, and express it with two multicoordinates $\aa=(\aa_p,\aa_q) \in \C^{ \mu }\times \C^{ \mu }$. Consider each $\aa_p$, $\aa_q$ as a column vector of length $\mu$. Let $B_0(z)$ denote the line vector $\lv 1 & z & \cdots  & z ^{ \mu-1}\rv$. Consider a pair of normalized polynomials, the quotient rational map and its derivative
\begin{eqnarray} p_{\aa}(z)\label{p} &=&  B_0(z) \aa_p + 0\cdot z^{ \mu }+ z ^{\mu+1}\\ q_{\aa}(z) &=& B_0(z) \aa _q + z ^{\mu} \\
f_{\aa}(z)&=& \dfrac {p_{\aa}(z)}{q_{\aa}(z)} =\dfrac{ B_0(z) \aa_p+ 0\cdot z^{ \mu }+ z ^{\mu+1}}{B_0(z) \aa _q + z ^{\mu} }\\
 g_{\aa}(z)\label{g} &:= &\dfrac{{\rm d}f_{\aa}}{{\rm d}z}(z)=\dfrac{p_{\aa}'q_{\aa}-q_{\aa}'p_{\aa}}{q_{\aa}^2}\\
 w_\aa(z)&:=&p_{\aa}'q_{\aa}-q_{\aa}'p_{\aa}=\text{Wronskian}(q_\aa,p_\aa)  \end{eqnarray}

 Notice that $w_\aa$ is a monic polynomial of degree $2\mu$. Thus
 $$w_\aa(z)=(z-c_1)\cdots(z-c_{2\mu}).$$

 Consider the sets $$
 \Omega:=\{\aa\in \C^{2\mu} \mid  resultant(p_\aa,q_\aa)\ne 0\}= \{\aa \mid \deg(f_\aa)=\mu+1\}$$
$$ \Omega':=\{ \aa\in \C^{2\mu} \mid resultant(p_\aa,q_\aa)\cdot  discri(w_\aa) \ne 0\}.$$

We have:
 \REFLEM{normalization}\quad
 1) For $\aa\in \Omega$, the point $\infty$ is not a critical point for every $f_\aa$, i.e. all critical points of $f_{\aa}$ are  in the finite plane.
\\
 2) For any rational map $f$ of degree $\mu+1$ such that the point $\infty$ is not a critical point, there is a unique M\"obius transformation $M$ such that $M\circ f=f_\aa$ for some $\aa$.
 Consequently the Schwarzian derivative map $\Omega\ni \aa \mapsto S_{f_\aa}$ is injective.\\
 3) $$ \Omega'=\{ \aa\in \C^{2\mu} \mid  discri(w_\aa) \ne 0\}.$$
 \ENDLEM
 \begin{proof} 1) and 2) follow from an easy algebraic manipulation.
 For 3) it is easy to check that if the polynomial $p_\aa'q_\aa-q_\aa'p_\aa$ has only simple roots then  $p_\aa$ and $ q_\aa$ are co-prime (the converse is not true).\end{proof}

Denote by $Poly_{2\mu}$  the $2\mu$-dimensional space of monic polynomials of degree at most $2\mu$.
Set $$\WWW:   \C^{2\mu}\ni \aa\mapsto w_\aa\in Poly_{2\mu} ,$$
it is called the {\it Wronskian operator}.
$$V:=\{\aa\in \C^{2\mu} \mid \WWW \text{ is locally invertible}\}=\{\aa\in \C^{2\mu} \mid Jac_\aa\WWW\ne 0\}$$
To every vector in $\C^{2\mu}$ we associate a monic polynomial
$$\C^{2\mu}\ni \cc=\lv c_1\\ \vdots \\ c_{2 \mu} \rv\mapsto R_\cc(z)=(z-c_1)\cdots(z-c_{2\mu})\in Poly_{2\mu} .$$
Permutations of the $c_i$'s will give the same polynomial. By the implicit function theorem
$$\cc\mapsto R_\cc$$ is locally invertible if and only if $c_i\ne c_j$ whenever $i\ne j$.

On the other hand, for any $\aa\in \Omega'$, the roots of $w_\aa$ are pairwise distinct. We may choose an ordering to turn the set of roots into a vector $\cc\in \C^{2\mu}$. There are $(2\mu)!$ choices of such an ordering.
\REFLEM{relate}\quad
The multivalued map $\Omega'\ni \aa\mapsto \cc\in \C^{2\mu}$ is locally invertible if and only if the single valued map  $\Omega'\ni \aa\mapsto w_\aa\in Poly_{2\mu}$ is locally invertible, i.e., $\aa\in \Omega'\cap V$.
\ENDLEM

\begin{example}\quad For $\mu =2$, $ \aa=(a_0,a_1,b_0,b_1)$, we have
$$w_\aa(z)=a_1 b_0-a_0 b_1-2 a_0 z+ \left(3 b_0-a_1\right)z^2 +2 b_1 z^3+z^4,\quad Jac_\aa\WWW=4 a_1 + 12 b_0$$
$$  V=\{\aa \mid 4 a_1 + 12 b_0\ne 0\}$$
 $$\Omega'=\{ \aa=(a_0,a_1,b_0,b_1) \mid  discri(w_\aa)\ne 0\}
$$
Let's us choose $a_1=b_1=b_0=0$ and $a_0=1$. In this case  $p_\aa(z)=1+z^3$ and $q_\aa(z)=z^2$. They are clearly coprime. So $\aa\in \Omega$.
We have $w_\aa(z)=z(z^3-2)$. All roots are simple. So $\aa\in \Omega'$.
As $Jac_\aa\WWW=0$,  we have $\aa\notin V$ and  $\aa\to \cc$ is not locally invertible. One checks easily that
$f_\aa$ has  four distinct critical values. This  example is actually equal to $$M(h_0\Big( \dfrac z {2^{1/3}}\Big) )$$
where the map $h_0$  is obtained at the beginning of this section with critical set $\{1,j,j^2,0\}$ and $M$ is a suitable M\"obius transformation.
\end{example}

\section{Counting}\label{count}

For $f$ a rational map, denote by $\CCC(f)$ the set of critical points of $f$.

Given any subset $\CCC$ of $\C$ consisting of $2(d-1)$ distinct points, and for $\AAA=(A_c)_{c\in \CCC}$ a collection of complex numbers, set
$$E_{m+1}(\AAA)= \sum_{c\in \CCC} (mc^{m-1} + c^m A_c),\ m\ge 0.$$

We have proved the following:

\REFLEM{equal}\quad Given any subset $\CCC$ of $\C$ consisting of $2(d-1)$ distinct points, the following sets are mutually bijective:
\begin{enumerate}
\item $\WWW^{-1}(\{\prod_{c\in \CCC}(z-c)\}), \ \Omega\cap \WWW^{-1}(\{\prod_{c\in \CCC}(z-c)\}),\ \Omega'\cap \WWW^{-1}(\{\prod_{c\in \CCC}(z-c)\})$
\item $\{ \text{ rational maps of degree d with critical set $\CCC$}\ \}/_{f\sim \text {\rm M\"obius\,}\circ f}$
\item $\{\aa\in \Omega\mid \CCC(f_\aa)=\CCC\}$
\item the set of quadratic differentials with doubles poles at $\CCC$, leading coefficients $\dfrac{-3}2$ and trivial holonomy group.
\item $\{\AAA=(A_c)_{ c\in \CCC}\mid  3A_c^2=4\ds \sum_{c'\in \CCC,c'\ne c } \left(\dfrac1{(c-c')^2}+ \dfrac{A_{c'}}{c-c'}\right), \forall\,c\in \CCC, \text{ and } E_j(\AAA)=0, j=1,2,3\}$
\item $\{\aa\in \Omega\mid \CCC(f_\aa)=M(\CCC)\}$, where $M$ is a M\"obius map with $M(\CCC)\subset \C$
\item  $\{ \text{ rational maps of degree d with critical set $M(\CCC)$}\ \}/_{f\sim \text {\rm M\"obius\,}\circ f}$.
\end{enumerate}
\ENDLEM

For example the sets 7 and 2 are related by $f\mapsto f\circ M$.

Using a highly non trivial result of Goldberg, we know that for a generic choice of $\CCC$ the cardinality of the above sets is the Catalan number $u_d=\dfrac1d\lv 2(d-1)\\ d-1\rv$. There are more elementary methods to prove that for a generic choice of $\CCC$, a lower bound of
the cardinality is $u_d$. For instance, A. Eremenko and A. Gabrielov \cite{EG2} provided an elementary way to construct $u_d$ different real rational maps with prescribed simple real critical points. Using this result, it is not difficult to prove that  for a generic choice of $\CCC$,  the  sets in Lemma \ref{equal} are
non-empty and have cardinality at least $u_d$.

For further  discussions on this counting problem, we refer to \cite{E1, EG2, EG3, G, S}.

The case $d=3$ with $u_3=2$ can be computed explicitly:
$$\WWW(\{a_0,a_1, b_0, b_1\})=\left\{a_1 b_0-a_0 b_1,-2 a_0,3 b_0-a_1,2 b_1\right\}.$$
When we try to solve $$\left\{a_1 b_0-a_0 b_1,-2 a_0,3 b_0-a_1,2 b_1\right\}=\{w_0,w_1,w_2,w_3\},$$
we get $$b_1=\dfrac{w_3}2,\ a_0=-\dfrac{w_1}2, \ a_1=3b_0-w_2,\ (3b_0-w_2)b_0+\dfrac{w_1w_3}4=w_0.$$
So $$a_1=\frac{1}{2} \left(-w_2\pm \sqrt{w_2^2+12 w_0-3 w_1 w_3}\right),\ b_0=\frac{1}{6}
   \left(w_2\pm \sqrt{w_2^2+12 w_0-3 w_1 w_3}\right). $$
So generically $\WWW^{-1}(  \{w_0,w_1,w_2,w_3\}) $ has two simple solutions, and it has a double solution if and only if $w_2^2+12 w_0-3 w_1 w_3=0$,  if and only if $a_1+3 b_0=0$.

Thus

\REFLEM{computation}\quad The set $\{\aa\mid a_1+3 b_0=0\}$ is the critical and the co-critical set of $\WWW$ and
$\{\ww \mid w_2^2+12 w_0-3 w_1 w_3=0\}$ is the critical value set of $\WWW$. The map
$$\WWW:\left\{\begin{array}{ll}\C^4\smm \{\aa\mid a_1+3 b_0=0\}\to \C^4\smm \{\ww \mid w_2^2+12 w_0-3 w_1 w_3=0\} & \text{ is a double covering}\\
  \{\aa\mid a_1+3 b_0=0\}\to  \{\ww \mid w_2^2+12 w_0-3 w_1 w_3=0\}& \text{ is a homeomorphism}.\end{array}\right.$$
\ENDLEM
We want to give a geometric interpretation of these sets in  terms of the critical points of  the rational map $f_\aa$, or the zeros of the polynomial $w_\aa(z)$.

\section{Geometry in the cubics}

For $w\ne 0$ with $\Im w>0$, or $\Im w=0$ and $\Re w> 0$, we consider $w,0,1,1+w$ as the 4 corners in cyclic order of a (eventually degenerate) parallelogram $P_w$.
\REFLEM{cross-ratio}\quad Given an ordered 4 distinct points $a,b,c,d\in \C$, we define its cross ratio
by $$[a,b,c,d]=\frac{(a-c)(b-d)}{(c-b)(d-a)}\ .$$
Then there is a unique M\"obius transformation $M$ mapping $a,c,b,d$ to the corners $w,0,1,1+w$ of  $P_w$,  in  order. \ENDLEM
\begin{proof}  Choose $w$ so that $w^2=[a,b,c,d]$ and either $\Im w>0$ or $\Im w=0$ and $\Re w\ge 0$.

Let $M$ be the unique M\"obius transformation sending $a,b,c$ to $w,1,0$. Then
$$[M(a), M(b),M(c), 1+w]=[w,1,0,1+w]=\dfrac{(w-0)(1-1-w)}{(0-1)\cdot (1+w-w)}=w^2=[a,b,c,d]\ .$$
It follows that $M(d)=1+w$.
\end{proof}

If $t=[a,b,c,d]$, then any permutation of the four points will give a cross ratio in the set
$$R(t)=\{t, \frac1t, 1-t,\frac1{1-t}, \frac t{t-1},\frac{t-1}t\}\ .$$
And any number in $R(t)$ is realized this way. For example, set $j:=e^{2\pi i/3}$, then $R(-j^2)=\{- j, - j^2\}$.

\begin{definition}\quad
We say that a set $\CCC$ of four distinct points  in $\cbar$ forms a {\bf regular tetrahedron} if there is a M\"obius map sending $\CCC$ to   $\{1,j,j^2,0\}$, where  $j:=e^{2\pi i/3}$.
\end{definition}

The motivation behind this definition is that the four points $1,j,j^2,0$ are located at the vertices of a regular ideal tetrahedron of the hyperbolic 3-ball. Any four distinct points on the Riemann sphere form
the vertices of an ideal tetrahedron (maybe degenerate). The collection of the angles between adjacent faces of the tetrahedron, or the 'shape' of the tetrahedron in W. Thurston's language, is an invariant of the M\"obius action. These angles are all equal precisely when the four points are M\"obius images of $\{1,j,j^2,0\}$.

\REFLEM{regular}\quad A set $\CCC$ of four distinct points  in $\cbar$ forms a {\bf regular tetrahedron} if and only
the cross ratio of an ordering of the set $\CCC$ belongs to $R(-j^2)$.\ENDLEM

\begin{proof} Note that $[1,j,j^2,0]=\dfrac{(1-j^2)j}{(j^2-j)(0-1)}=1+j=-j^2$  and M\"obius maps preserve cross ratios.
\end{proof}

\REFLEM{symmetric group}\quad Given any set $V$ of four distinct points on $\cbar$ there is a unique 4-group  $\MMM_V$ of M\"obius transformations,  isomorphic of $\Z/2\Z \times \Z/2\Z$, such that every non-identity element of $\MMM_V$ acts on $V$ as an order 2
permutation without fixed points. The map $V\mapsto \MMM_V$ is not injective. The group of M\"obius transformations preserving   the set $V$ may
be strictly larger than $\MMM_V$.
\ENDLEM

\begin{proof} Assume $V=\{a,b,c,d\}$. Let $M_{a,b}$ be the unique M\"obius transformation
mapping  $a,b,c$ to $b,a,d$ respectively. As $[b,a,d,c]=[a,b,c,d]=[M(a), M(b), M(c), M(d)]$ we have necessarily $M(d)=c$.

Here is a geometric proof. Choose $w$ so that $w^2=[a,b,c,d]$ and either $\Im w>0$ or $\Im w=0$ and $\Re w\ge 0$.

Let $N$ be the unique M\"obius transformation sending $a,b,c,d$ to $w,1,0,1+w$.
The vertical line from $\dfrac{1+w}2$ to $\infty$
in the upper 3-space is the unique common perpendicular of the two diagonal geodesics $w\leftrightarrow 1$ and $0 \leftrightarrow 1+w$.
A rotation of 180 degree switches $w$ and $1$, and switches $0$ and $1+w$.
There is therefore a unique common perpendicular of the two  geodesics $a\leftrightarrow b$ and $c \leftrightarrow d$. A rotation of 180 degree with respect to this perpendicular is our desired map $M_{a,b}$.

Now the 4-group is simply $\{id, M_{a,b}, M_{a,c}, M_{a,d}\}$.\end{proof}

\begin{lemma}[{Buff}]\label{Xavier}\quad
 Two rational maps $f$ and $g$ having the same set of critical points and the same images for each of them, are in fact equal (except maybe in the bicritical case).\end{lemma}

\begin{proof}
Let $\CCC$ be the common critical set and $d$ be the common degree of $f$ and $g$.

On the one hand, each point in $\CCC$ is a critical point of $f/g$ with multiplicity $\ge$ to the multiplicity as a critical point of $f$.
The image of such a point is $1$.
Thus, $1$ has at least $2d-2+ |\CCC|$ preimages counting multiplicities.

On the other hand, $f/g$ is a rational map of degree $\le 2d$. Thus,
$2d-2+|\CCC| \le 2d$ and thus, $|\CCC|\le 2$.\end{proof}

For a rational map $f$, denote by $\CCC(f)$, resp. $\VVV(f)$, the set of critical points, resp. critical values, of $f$.
The following was a remark of W. Thurston.

\REFLEM{cubic lift}\quad For any cubic rational map $f$ with four distinct critical points and four distinct critical values, there is
a unique bijection  $M\leftrightarrow N$ between $\MMM_{\CCC(f)}$ and   $\MMM_{\VVV(f)}$
such that $f\circ M=N\circ f$.  \ENDLEM
\begin{proof} Let   $M\in \MMM_{\CCC(f)}$ ,  $N\in \MMM_{\VVV(f)}$ be two maps so that $M$ permutes the critical points in the same way as
$N$ on the corresponding critical values. Then
$f\circ M|_{\CCC(f)}=N\circ f|_{\CCC(f)}$. So the two rational maps $f\circ M$ and $N\circ f$
have the same action on a identical critical set. By Lemma \ref{Xavier} this implies that they are equal.
\end{proof}

The following statement was first suggested to us by W. Thurston in 2011. Actually it was already proved by Lisa Goldberg in [G], Theorem 1.4, as a consequence of her deep counting result. Here we provide an elementary proof using Lemma \ref{computation}.
\begin{theorem}[{Goldberg}]\label{invariance}
A vector $\aa\in \Omega'$  is a critical point of the Wronskian operator $\WWW$
if and only if the set $\CCC(f_\aa)$ of critical points of $f_\aa$ forms a regular tetrahedron.
A polynomial in $Poly_{4}$ with 4 distinct roots is a critical value of the Wronskian operator $\WWW$
if and only if the set of its roots forms a regular tetrahedron.
\end{theorem}
\begin{proof}
Recall the family defined above:
$$h_{\alpha}(z)=\dfrac{{\alpha}(z^3+2)+3z^2}{3{\alpha}z+2z^3+1}.$$

Set $\CCC=\{1,j,j^2,0\}$.
Let $\CCC'=\{s,t,u,v\}$ be a set of 4 distinct points in $\C$.

If $[s,t,u,v]\in R(-j^2)$, there is a M\"obius map $M$ such that $M(\CCC)=\CCC'$.
By Lemma \ref{equal}.6
$$1=\#\WWW^{-1}(\{\prod_{c\in \CCC}(z-c)\}) = \#\WWW^{-1}(\{\prod_{c\in \CCC'}(z-c)\}).$$
It follows that the polynomial $\prod_{c\in \CCC'}(z-c)$ is a critical value of the operator $\WWW$.

If $[s,t,u,v]\notin R(-j^2)$, then for the unique M\"obius map $N$ sending $s,t,u$ to $1,j,j^2$, we have $N(v)\in \C\smm \{0, 1,j,j^2\}$. Set $\al^2=N(v)$. There are two maps $h_\al$ and $h_{-\al}$
realizing $\{1,j,j^2,\al^2\}$ as critical set, and $h_\al$ and $h_{-\al}$ do not differ by a post composition of a M\"obius map. We have
$$2\le \#\{ \text{cubic rat. maps $f$ with $\CCC(f)=\{1,j,j^2,\al^2\}$}\}/_{f \sim\text{\rm M\"obius} \circ f}$$
$$\overset{\text{Lemma \ref{equal}.2}}=\#\WWW^{-1}(\{\prod_{c\in \{1,j,j^2,\al^2\}}(z-c)\}) \overset{\text{Lemma \ref{equal}.6} }= \#\WWW^{-1}(\{\prod_{c\in \CCC'}(z-c)\})\overset{\text{Lemma \ref{computation}}}\le 2$$
It follows that $\#\WWW^{-1}(\{\prod_{c\in \CCC'}(z-c)\})=2$ and
the polynomial $\prod_{c\in \CCC'}(z-c)$ is not a critical value of the operator $\WWW$.

Similarly, for $\aa\in \Omega'$ so that $\CCC(f_\aa)$ has 4 distinct critical points. Then
$$\WWW^{-1}(\{\prod_{c\in \CCC(f_\aa)}(z-c)\})$$ has 1 or 2 points depending whether $\CCC(f_\aa)$
forms a regular tetrahedron or not. In the former case, $\aa$ is a critical point of  the operator $\WWW$.
Otherwise $\aa$ is a regular point of $\WWW$.\end{proof}

We may express the result in algebraic terms as follows. \REFCOR{algebraic}\quad
Let $z^4+w_3 z^3+w_2 z^2 + w_1 z +w_0$ be a quartic polynomial with 4 distinct roots. Denote by $\De$ their cross ratio. Then $\De\in R(-j^2)$ if and only if $w_2^2+12 w_0-3 w_1 w_3=0$.
Let $(z^3+ a_1z + a_0, z^2+b_1z+b_0)$ be a pair of polynomials whose wronskian has
4 distinct roots, or equivalently whose ratio as a rational function has 4 distinct critical points. Denote by $\De$ the cross ratio of these four points. Then $\De\in R(-j^2)$ if and only if
$a_1+3 b_0=0$.
\ENDCOR

The corollary together with Lemma \ref{computation} gives Theorem \ref{dependency}.

Note that the M\"obius symmetric group of any four distinct points is non-trivial (Lemma \ref{symmetric group}). Thus for any cubic rational map $f$ with four distinct critical values (hence with four distinct critical points), there exist non-trivial M\"obius maps $M$ and $N$ such that  $f\circ M=N\circ f$ (see Lemma \ref{cubic lift}).

In the case $d\ge 4$, a generic set of $2d-2$ distinct points has a trivial M\"obius symmetric group.
We suspect that the critical points of $\WWW$ occur precisely at maps $f_\aa$ such that there are
non-trivial M\"obius maps $M,N$ with $f_\aa\circ M=N\circ f_\aa$.

For a holomorphic map from an open set of $\C^k$ into $\C^k$, its Jacobian is non zero if and only if it is locally injective (see \cite[Thm. 7.1]{F}). One may use this fact to prove that the critical set and and critical value set of $\WWW$ are both M\"obius invariant, in the following sense.

\REFLEM{higher degree}\quad Let $\mu\ge 2$.
Consider two sets of  $2\mu$ distinct points $\CCC$ and $\CCC'$ in $\C$  such that there is a M\"obius map $M$ satisfying $M(\CCC)=\CCC'$.
Then the polynomial with roots $\CCC$ is a critical value of $\WWW$ if and only if  the polynomial with roots $\CCC'$ is a critical value of $\WWW$.

Let $\aa\in \C^{2\mu}$ such that the rational map $f_\aa$ has critical point set $\CCC$. Then there is a unique M\"obius map $N=N_\aa$ and a unique $\aa'$ such that $N\circ f_\aa\circ M^{-1}=f_{\aa'}$.
And $\aa$ is a critical point of $\WWW$ if and only if $\aa'$ is. \ENDLEM

One can use this lemma to normalize three critical points and three critical values, and thus reduce
our study of $\WWW$ to a subvariety of dimension $2\mu-3$, just as what we did in the cubic case.

\section{Problems for future developments}

In sections 6, 7, 8, we studied how the Schwarzian derivative $S_f$ of a rational map $f$ with only simple critical points depends upon the poles of $S_f$ itself. We showed that $S_f$ does not depend holomorphically on its poles if and only if the poles are the zeros of a critical value polynomial of the Wronskian operator (see Section 6). Furthermore, we gave both algebraic and geometric characterizations of the critical points and critical value polynomials of the Wronskian operator in the case of $d=3$.

For a future research, a natural question is to extend this study of the Wronskian operator to higher degrees. More precisely, in the algebraic aspect, we want to find an explicit formula (using the coefficients of $p_\aa$ and $q_\aa$) defining the critical locus of the Wronskian operator; while in the geometric aspect, we wish to  characterize the critical points of the Wronskian operator by the symmetry of the rational maps: we conjecture that the critical points of $\WWW$ occur precisely at rational maps $f$ such that there are non-trivial
M\"obius maps $M,N$ satisfying $f\circ M=N\circ f$.

In early sections, we gave necessary and sufficient conditions for a local meromorphic function $\phi$ to admit a
local meromorphic Schwarzian primitive $f$. This map $f$ may have a critical point of any multiplicity.
We then gave, in section $5$,  a global description of the Schwarzian derivatives of the rational maps with only simple critical points. A natural sequel would be to carry out  the global study to rational maps with multiple critical points. In this more general setting, questions about holomorphic dependencies on poles may  still be addressed.

Another fascinating direction of further research is to study the trajectory structure of Schwarzian derivatives, as
a particular class of quadratic differentials. During autumn 2010, encouraged by W. Thurston, the last two authors of this present article did quite a lot of numerical experiences on the cubic case. Part of their observations have been proved in the sub sequel group discussion led by Thurston. But many things remain to be proved and higher degree cases remain to be exploited.

\Acknowledgements{This note is based  on  a group discussion lead by W. Thurston in autumn 2010, and on suggestions of J.H. Hubbard and R. Schaefke. The motivation can be found in Thurston's post on Mathoverflow \cite{T}. The determinant in the definition of $Y_d$ was suggested by R. Schaefke. We would like to thank L. Bartholdi, X. Buff, A. Epstein and M. Loday for helpful discussions. The first author is supported by the grant no. 11125106 of NSFC.
He thanks the hospitality of the Angers University where the major part of this research took place. The second author is supported by the grant no. 11501383 of NSFC. The last author is supported by project LAMBDA, ANR-13-BS01-0002.}


\end{document}